\documentclass{amsart}
\usepackage{a4wide}
\usepackage{amssymb}
\usepackage{amsfonts,amsmath,amsthm}

\newcommand{\IR}{\mathbb R}
\newcommand{\br}{\mathbf r}
\newcommand{\bp}{\mathbf p}
\newcommand{\bs}{\mathbf s}
\newcommand{\bt}{\mathbf t}
\newcommand{\e}{\varepsilon}
\newcommand{\sign}{\mathrm{sign}}
\newcommand{\jj}{\dot\jmath}
\newcommand{\jjj}{\ddot\jmath}
\newcommand{\K}{\mathcal K}
\newcommand{\hS}{\breve S}

\newtheorem{theorem}{Theorem}[section]
\newtheorem{corollary}[theorem]{Corollary}
\newtheorem{lemma}[theorem]{Lemma}

\newtheorem{problem}[theorem]{Problem}
\newtheorem{proposition}[theorem]{Proposition}

\theoremstyle{definition}
\newtheorem{definition}[theorem]{Definition}
\newtheorem{remark}[theorem]{Remark}

\title[Every non-smooth $2$-dimensional Banach space has the Mazur--Ulam property]{Every non-smooth $2$-dimensional Banach space has\\ the Mazur--Ulam property}
\author{Taras Banakh and Javier Cabello S\'anchez}
\address{T.Banakh: Ivan Franko National University of Lviv (Ukraine) and Jan Kochanowski University in Kielce (Poland)}
\email{t.o.banakh@gmail.com}
\address{J.~Cabello S\'anchez: Departamento  de  Matem\'aticas  and  Instituto  de  Matem\'aticas.   Universidad  de  Extremadura, Avda.  de Elvas s/n, 06006, Badajoz, Spain}
\email{oco@unex.es}

\begin{document}
\begin{abstract} A Banach space $X$ has {\em the Mazur--Ulam property} if any  isometry from the unit sphere of $X$ onto the unit sphere of any other Banach space $Y$ extends to a linear isometry of the Banach spaces $X,Y$.
A Banach space $X$ is called {\em smooth} if the unit ball has a unique supporting functional at each point of the unit sphere. We prove that each non-smooth 2-dimensional Banach space has the Mazur--Ulam property.
\end{abstract}
\maketitle

\section{Introduction}

By the classical result of Mazur and Ulam \cite{MU}, every bijective isometry  between Banach spaces is affine. This result essentially asserts that the metric structure of a Banach space determines its linear structure. In  \cite{Man} Mankiewicz proved that every bijective isometry $f:B_X\to B_Y$ between the unit balls of two Banach spaces $X,Y$ extends to a linear isometry of the Banach spaces. In \cite{Tingley} Tingley asked if the unit balls in this result of Mankiewicz can be replaced by the unit spheres. More precisely, he posed the following problem.

\begin{problem}[Tingley, 1987]\label{prob:Tingley} Let $f:S_X\to S_Y$ be a bijective isometry between the unit spheres of two Banach spaces $X,Y$. Can $f$ be extended to a linear isometry between the Banach spaces $X,Y$?
\end{problem}

Here for a Banach space $(X,\|\cdot\|)$ by 
$$B_X=\{x\in X:\|x\|\le 1\}\quad\mbox{and}\quad S_X=\{x\in X:\|x\|=1\}$$we denote the unit ball and unit sphere of $X$, respectively.

Tingley's Problem~\ref{prob:Tingley} remains open even for 2-dimensional Banach spaces. This problem can be equivalently reformulated in terms of the Mazur--Ulam property,  introduced by Cheng and Dong \cite{CD} and widely used in the the literature devoted to Tingley's problem \cite{FPP, Per, Tan}. 

\begin{definition} A Banach space $X$ is defined to have the {\em Mazur--Ulam property} if every  isometry $f:S_X\to S_Y$ of $S_X$ onto the unit sphere $S_Y$ of an arbitrary Banach space $Y$ extends to a linear isometry of the Banach spaces $X,Y$.
\end{definition} 

In fact, Tingley's Problem~\ref{prob:Tingley} asks whether every Banach space has the Mazur--Ulam property. There are many results on the Mazur--Ulam property in some special Banach spaces like $C(K)$, $c_0(\Gamma)$, $\ell_p(\Gamma)$, $L_p(\mu)$, see the survey \cite{Per}. By a result of Kadets and Mart\'\i n \cite{KM}, every polyhedral finite-dimensional Banach space has the Mazur--Ulam property. 

For 2-dimensional Banach spaces this result of Kadets and Mart\'\i n was improved by Cabello S\'anchez who proved the following theorem in \cite{San}.

\begin{theorem}\label{t:CS} A $2$-dimensional Banach space  has the Mazur--Ulam property if  is not strictly convex.
\end{theorem}

Let us recall that a Banach space $X$ is {\em strictly convex} if each convex subset of the unit sphere $S_X$ contains at most one point.

A Banach space $X$ is {\em smooth} if for every point $x\in S_X$ there exists a unique linear continuous functional $x^*:X\to \IR$ such that $x^*(x)=1=\|x^*\|$. Geometrically this means that the unit ball $B_X$ has a unique supporting hyperplane at $x$. 

It is well-known \cite[7.23]{FA} that a reflexive Banach space $X$ is strictly convex if and only if its dual Banach space $X^*$ is smooth.

The main result of this paper is the following theorem, a kind of a dual version of Theorem~\ref{t:CS}.

\begin{theorem}\label{t:main} Each non-smooth $2$-dimensional Banach space has the Mazur--Ulam property.
\end{theorem}

This theorem follows from Propositions~\ref{p1} and \ref{p2}, proved in Section~\ref{s:special}.  For piecewise $C^1$-smooth Banach spaces, Theorem~\ref{t:main} was proved in \cite{CS}.

\begin{remark} Theorems~\ref{t:main} and Proposition~\ref{p2} (on the Mazur--Ulam property of $2$-dimensional Banach spaces whose sphere contains two linearly independent special directions) are essential ingredients in the main result of the paper \cite{Ban2} answering the Tingley's Problem in the class of $2$-dimensional Banach spaces.
\end{remark}

\section{The natural parameterization of the unit sphere of a 2-dimensional Banach space}

By a {\em $2$-based Banach space} we understand any 2-dimensional Banach space $X$ endowed with a basis $\mathbf e_1,\mathbf e_2$ consisting of two linearly independent vectors in $X$. 

Let $X$ be a 2-based Banach space and $\mathbf e_1,\mathbf e_2$ be the basis of $X$.

 {\em The polar parameterization} of the unit sphere $S_X$ is the map $$\mathbf p:\IR\to S_X,\quad\mathbf p:t\mapsto \frac{\mathbf e^{it}}{\|\mathbf e^{it}\|},\quad\mbox{where}\quad\mathbf e^{it}=\cos(t)\mathbf e_1+\sin(t)\mathbf e_2.$$
 
The following properties of the polar parameterization were established in \cite[\S4]{Ban}.

\begin{lemma}\label{l:p} The polar parameterization $\bp:\IR\to S_X$ has the following properties:
\begin{enumerate}
\item $\bp(t+\pi)=-\bp(t)$ for every $t\in\IR$;
\item the function $\bp$ has one-sided derivatives $$\bp'_-(t)=\lim_{\e\to-0}\frac{\bp(t+\e)-\bp(t)}{\e}\mbox{ \ and \ }\bp'_+(t)=\lim_{\e\to+0}\frac{\bp(t+\e)-\bp(t)}{\e}$$ at each point $t\in\IR$;
\item the set $\{t\in \IR:\bp'_-(t)\ne\bp'_+(t)\}$ is at most countable.
\item $\dfrac{c}{C}\cdot |\sin(\e)|\le \|\bp(t+\e)-\bp(t)\|\le\dfrac{2C^2}{c^2}\cdot |\e|$ for any $t,\e\in\IR$;
\item $\dfrac{c}{C}\le \min\{\|\bp'_-(t)\|,\|\bp'_+(t)\|\}\le 
 \max\{\|\bp'_-(t)\|,\|\bp'_+(t)\|\}\le \dfrac{2C^2}{c^2}$  for every $t\in\IR$, 
 \end{enumerate}
 where
$ c=\min\{\|\mathbf e^{it}\|:t\in\IR\}$ and $C=\max\{\|\mathbf e^{it}\|:t\in\IR\}$.
\end{lemma}

Lemma~\ref{l:p} implies that the function 
$$\bs:\IR\to\IR,\;\;\bs:t\mapsto \int_0^t\|\bp'_{-}(u)\|\,du=\int_0^t\|\bp'_+(u)\|du$$
is continuous and strictly increasing.

For $t\in[0,\pi]$ the value $\bs(t)$ can be thought as the length of the curve on the sphere $S_X$ between the points $\bp(0)$ and $\bp(t)$ in the Banach space $X$.
The number
$$L=\bs(\pi)=\int_0^\pi\|\bp'_-(t)\|dt=\int_0^\pi\|\bp'_+(t)\|dt$$
is called the {\em half-length} of the sphere $S_X$ in $X$.

The image $$\breve S_X=\{\mathbf p(t):0\le t\le\pi\}$$ is called the {\em upper half-sphere} of the $2$-based Banach space $X$.

Since the function $\bs$ is continuous and increasing, there exists a unique continuous increasing function $\bt:\IR\to \IR$ such that $\bs\circ\bt $ is the identity map of $\IR$.

The function
$$\br:\IR\to S_X,\;\br:s\mapsto \bp\circ \bt(s),$$ is called {\em the natural parameterization} of the sphere $S_X$. 

The following properties of the natural parameterization were established in \cite[\S5]{Ban}.

\begin{lemma}\label{l:r} The natural parameterization $\br:\IR\to S_X$ of $S_X$ has the following properties:
\begin{enumerate}
\item $\br(s+L)=-\br(s)$ for every $s\in\IR$;
\item the function $\br$ has one-sided derivatives $$\br'_-(s)=\lim_{\e\to-0}\frac{\br(s+\e)-\br(s)}{\e}\mbox{ \ and \ }\br'_+(s)=\lim_{\e\to+0}\frac{\br(s+\e)-\br(s)}{\e}$$ at each point $s\in\IR$;
\item the set $\{s\in \IR:\br'_-(s)\ne\br'_+(s)\}$ is at most countable;
\item$\br$ is non-expanding and has $\|\br'_-(s)\|=\|\br'_+(s)\|=1$ for every $s\in\IR$.
\item If $\br$ is differentiable on some open set $U\subseteq \IR$, then $\br$ is continuously differentiable on $U$.
\end{enumerate}
\end{lemma}

The natural parametrization is closely related to the intrinsic metric on the half-sphere $\breve S_X$.

For two points $x,y\in \breve S_X$ the real number 
$$\breve d(x,y)=\sup_{\e>0}\inf\Big\{\sum_{i=1}^n\|x_{i+1}-x_i\|:x_0,\dots,x_n\in\breve S_X,\;x_0=x,\;x_n=y,\;\max_{1\le i\le n}\|x_{i+1}-x_i\|<\e\Big\}$$
is called the {\em intrinsic distance}  between the points $x,y$ on the half-sphere $\breve S_X$.
The following lemma can be proved by analogy with Lemma~3.1 of \cite{Ban}.

\begin{lemma}\label{l:intrinsic} If $\br$ is continuously differentiable at each point $s\in(0,L)$, then the map $\br{\restriction}_{[0,L]}$ is an isometry of the interval $[0,L]$ onto the half-sphere $\hS_X$ endowed with the intrinsic distance.
\end{lemma}

If $\br$ is arbitrary, then we can prove a weaker statement.

\begin{lemma}\label{l:d} For any $s\in\IR$ and small  $\e$ we have $$\|\br(s+\e)-\br(s)\|=(1+o(1))\cdot|\e|.$$
\end{lemma}

\begin{proof} If $\e>0$, then $\br(s+\e)=\br(s)+\br'_+(s)\e+o(\e)$ and hence
$$\|\br(s+\e)-\br(s)\|=\|\br'_+(s)+o(1)\|\cdot|\e|=(\|\br'_+(s)\|+o(1))\cdot|\e|=(1+o(1))\cdot|\e|.$$
By analogy we can show that $\|\br(s+\e)=\br(s)\|=(1+o(1))\cdot|\e|$ for small $\e<0$.
\end{proof}

For every parameter $s\in \IR$ let $$\br'_\pm(s)=\tfrac12\big(\br'_+(s)+\br'_-(s)\big).$$
It is easy to see that the vectors $\br(s)$ and $\br'_\pm(s)$ are linearly independent. Then
$$\tfrac12(\br'_+(s)-\br'_-(s))=\jj (s)\cdot \br(s)+\jjj (s)\cdot\br'_\pm(s)$$
for some real numbers $\jj(s)$ and $\jjj(s)$, called the {\em radial} and {\em tangential jumps} of the derivative $\br'$ at $s$, respectively.

It follows that 
$$\br'_+(s)=\jj(s)\cdot\br(s)+(1+\jjj(s))\cdot\br'_\pm(s)\quad\mbox{and}\quad
\br'_-(s)=-\jj(s)\cdot\br(s)+(1-\jjj(s))\cdot\br'_\pm(s).$$

\begin{lemma}\label{l:j}
\begin{enumerate}
\item $|\jjj(s)|<1$.
\item $\br'_-(s)=\br'_+(s)$ iff $\jj(s)=0$.
\item $\jj(s)\le 0$.
\end{enumerate}
\end{lemma}

\begin{proof} 1. It is easy to see that the bases $\big(\br(s),\br'_-(s)\big)$, $\big(\br(s),\br'_+(s)\big)$, $\big(\br(s),\br'_\pm(s)\big)$ have the same orientation, which implies that for the basis $\big(\br(s),\br'_\pm(s)\big)$ the $\br'_\pm(s)$-coordinates $1+\jjj(s)$ and $1-\jjj(s)$ of the vectors $\br'_+(s)$ and $\br'_-(s)$  are positive and hence $|\jjj(s)|<1$.
\smallskip

2. If $\br'_-(s)=\br'_+(s)$, then $0=\frac12(\br'_+(s)-\br'_-(s))=\jj(s)\cdot \br(s)+\jjj(s)\cdot\br'_\pm(s)$ and hence $\jj(s)=0$. If $\jj(s)=0$, then the vectors $\br'_+(s)=(1+\jjj(s))\br'_\pm(s)$ and $\br'_-(s)=(1-\jjj(s))\br'_\pm(s)$ are collinear and hence they are equal since they have the same norm and the bases $(\br(s),\br'_-(s))$ and $(\br(s),\br'_+(s))$ have the same orientation.
\smallskip
\newpage

3. The inequality $\jj(s)\le 0$ follows from the convexity of the ball $B_X$, see the following picture.

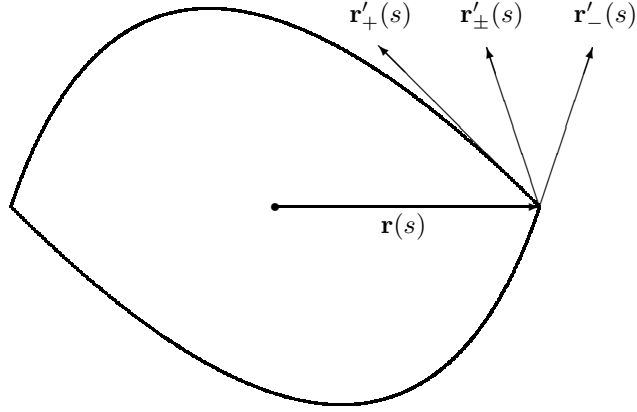
\begin{picture}(200,180)(-100,15)
\put(100,100){\circle*{3}}
\put(100,100){\vector(1,0){100}}
\put(200,100){\vector(1,3){20}}
\put(200,100){\vector(-1,1){61}}
\put(200,100){\vector(-1,3){20}}


\put(140,90){$\br(s)$}
\put(128,170){$\br'_+(s)$}
\put(170,170){$\br'_\pm(s)$}
\put(213,170){$\br'_-(s)$}

\thicklines
\qbezier(0,100)(50,250)(200,100)
\qbezier(0,100)(150,-50)(200,100)
\end{picture}
\end{proof}

In the following lemma we use the standard function $\sign:\IR\to\{-1,0,1\}$ defined by the formula
$$
\sign(\e)=\begin{cases}
1&\mbox{if $\e>0$};\\
0&\mbox{if $\e=0$};\\
-1&\mbox{if $\e<0$}.
\end{cases}
$$

\begin{lemma}\label{l:a} Let $a,b,s\in\IR$ be such that $\br$ is differentiable at $b$ and $$0\ne \br(b)-\br(a)=\|\br(b)-\br(a)\|\cdot \br(s).$$ Let $\br'(b)=x\cdot \br(s)+y\cdot \br'_\pm(s)$ for some $x,y\in\IR$ with $y\ne 0$. Then $y>0$ and for any small $\e$ we have the asymptotic formula
$$\|\br(b+\e)-\br(a)\|=\|\br(b)-\br(a)\|+x\cdot\e-\frac{\sign(\e)\cdot \jj(s)\cdot y}{1+\sign(\e)\cdot\jjj(s)}\cdot\e+o(\e).$$
\end{lemma}

\begin{proof} The positivity of $y$ follows from $0\ne\br(b)-\br(a)=\|\br(b)-\br(a)\|\cdot \br(s)$ and the same orientation of the bases $(\br(s),\br'_\pm(s))$ and $(\br(s),\br'(b))$, see the following picture.

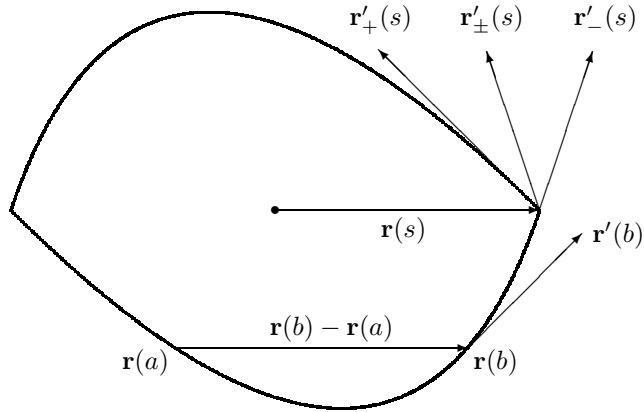
\begin{picture}(200,200)(-100,0)
\put(100,100){\circle*{3}}
\put(100,100){\vector(1,0){100}}
\put(200,100){\vector(1,3){20}}
\put(200,100){\vector(-1,1){61}}
\put(200,100){\vector(-1,3){20}}


\put(140,90){$\br(s)$}
\put(128,170){$\br'_+(s)$}
\put(170,170){$\br'_\pm(s)$}
\put(213,170){$\br'_-(s)$}
\put(220,88){$\br'(b)$}
\put(98,53){$\br(b)-\br(a)$}

\put(170,45){\vector(1,1){46}}
\put(63,48){\vector(1,0){110}}

\put(42,40){$\br(a)$}
\put(175,40){$\br(b)$}

\thicklines
\qbezier(0,100)(50,250)(200,100)
\qbezier(0,100)(150,-50)(200,100)
\end{picture}

Since $\br(b)-\br(a)=\|\br(b)-\br(a)\|\cdot\br(s)$, for a small $\e$ there exists a small $\delta$ such that 
$$\br(b+\e)-\br(a)=\|\br(b+\e)-\br(a)\|\cdot \br(s+\delta).$$
By the differentiability of $\br$ at $b$, we obtain
$$
\begin{aligned}
\br(b+\e)-\br(a)&=\br(b)+\br'(b)\e+o(\e)-\br(a)=\|\br(b)-\br(a)\|\cdot\br(s)+(x{\cdot}\br(s)+y{\cdot}\br'_\pm(s)){\cdot}\e+o(\e)=\\
&=\big(\|\br(b)-\br(a)\|+x\e+o(\e))\cdot\br(s)+(y\e+o(\e))\cdot\br'_\pm(s).
\end{aligned}
$$
If $\delta>0$, then 
$$
\br(s+\delta)=\br(s)+\br'_+(s)\delta+o(\delta)=(1+(\jj(s)+o(1))\delta){\cdot}\br(s)+(1+\jjj(s)+o(1))\delta{\cdot}\br'_\pm(s).
$$
If $\delta<0$, then 
$$
\br(s+\delta)=\br(s)+\br'_-(s)\delta+o(\delta)=(1-(\jj(s)+o(1))\delta){\cdot}\br(s)+(1-\jjj(s)+o(1))\delta{\cdot}\br'_\pm(s).
$$

In both cases we obtain
$$
\br(s+\delta)=(1+(\sign(\delta){\cdot}\jj(s)+o(1))\delta){\cdot}\br(s)+(1+\sign(\delta){\cdot}\jjj(s)+o(1))\delta{\cdot}\br'_\pm(s).
$$
It follows that
$$
\begin{aligned}
&(\|\br(b)-\br(a)\|+x\e+o(\e))\br(s)+(y\e+o(\e))\br'_\pm(s)=\br(b+\e)-\br(a)=\\
&=\|\br(b+\e)-\br(a)\|\cdot\br(s+\delta)=\\
&=\|\br(b+\e)-\br(a)\|\cdot \big(1+(\sign(\delta){\cdot}\jj(s)+o(1))\delta){\cdot}\br(s)+(1+\sign(\delta){\cdot}\jjj(s)+o(1))\delta{\cdot}\br'_\pm(s)\big).
\end{aligned}
$$
Writing this equation in coordinates, we obtain two equations:
\begin{equation}\label{eq:one}
\|\br(b+\e)-\br(a)\|\cdot \big(1+(\sign(\delta){\cdot}\jj(s)+o(1))\delta\big)=\|\br(b)-\br(a)\|+x{\cdot}\e+o(\e)
\end{equation} 
and
\begin{equation}\label{eq:two}
\|\br(b+\e)-\br(a)\|\cdot (1+\sign(\delta){\cdot}\jjj(s)+o(1))\cdot \delta=y\cdot\e+o(\e).
\end{equation} 
The equation (\ref{eq:two}) implies $$\delta=\frac{(y+o(1))\e}{\|\br(b+\e)-\br(a)\|(1+\sign(\delta){\cdot}\jjj(s))}=\frac{(y+o(1))\e}{\|\br(b)-\br(a)\|\cdot (1+\sign(\delta){\cdot}\jjj(s))}.$$Since $y>0$ and $|\jjj(s)|<1$, this implies 
$$\sign(\delta)=\sign(\e)+o(1).$$
After substitution of $\delta$ into the equation (\ref{eq:one}), we obtain
$$
\begin{aligned}
&\|\br(b+\e)-\br(a)\|=\frac{\|\br(b)-\br(a)\|+x{\cdot}\e+o(\e)}{1+(\sign(\delta){\cdot}\jj(s)+o(1))\delta}=\\
&=\big(\|\br(b)-\br(a)\|+x{\cdot}\e+o(\e)\big)\cdot\big(1-(\sign(\delta){\cdot}\jj(s)+o(1))\delta\big)=\\
&=\big(\|\br(b)-\br(a)\|+x{\cdot}\e+o(\e)\big)\cdot\Big(1-\frac{(\sign(\e){\cdot}\jj(s)+o(1))(y+o(1))\cdot \e}{\|\br(b)-\br(a)\|(1+\sign(\e){\cdot}\jjj(s))}\Big)=\\
&=\|\br(b)-\br(a)\|+x\cdot\e-\frac{\sign(\e){\cdot}\jj(s){\cdot}y}{1+\sign(\e){\cdot}\jjj(s)}\cdot\e+o(\e).
\end{aligned}
$$

\end{proof}

Lemmas~\ref{l:a} and \ref{l:j} imply the following corollary.

\begin{corollary}\label{l:dif} Let $a,b,s\in\IR$ be real numbers such that the map $\br$ is differentiable at $b$ and $$0\ne \br(b)-\br(a)=\|\br(b)-\br(a)\|\cdot \br(s).$$ The map $\br$ is differentiable at the point $s$ if and only if the map 
$$\nu:\IR\to \IR,\quad\nu:\e\mapsto \|\br(b+\e)-\br(a)\|,$$
is differentiable at zero.
\end{corollary}

\section{Recognizing smooth points on the unit sphere}

A point $x$ of the unit sphere of a Banach space $X$ is defined to be {\em smooth} if there exists a unique linear continuous functional $x^*:X\to\IR$ such that $x^*(x)=1=\|x^*\|$. 

The following lemma proved by Cabello S\'anchez \cite{CS} shows that the smooth points can be recognized by measurements of distances on the sphere.

\begin{lemma}\label{l:ns} A point $p\in S_X$ is non-smooth if and only if there exists positive real numbers $\delta$ and $\e_0$ such that for every positive $\e<\e_0$ there are points $x,y\in S_X$ such that $$\max\{\|x-p\|,\|y+p\|\}<\e\quad\mbox{and}\quad\|u-v\|<(2-\delta)\e.$$
\end{lemma}

Another smoothness criterion is given by the following lemma.

\begin{lemma}\label{l:sd} Let $X$ be a strictly convex $2$-dimensional Banach space and $a,b,c\in S_X$ be points such that $\br$ is differentiable at $b$ and $0\ne b-a=\|b-a\|\cdot c$. The point $c$ is not smooth if and only if there exist positive real numbers $\delta$ and $\e_0$ such that for every positive $\e<\e_0$ there are two distinct points $x,y\in S_X$ such that $\|x-b\|=\e=\|y-b\|$ and $\|x-a\|+\|y-a\|>2{\cdot}\|b-a\|+\delta\cdot \e$.
\end{lemma}

\begin{proof} Assume that the point $b$ is smooth and fix any basis $\mathbf e_1,\mathbf e_2$ for the space $X$ such that $\mathbf e_1=c$. Let $\br:\IR\to S_X$ be the natural parameterization of the 2-based Banach space $(X,\mathbf e_1,\mathbf e_2)$. Then $\br(0)=\mathbf e_1=c$. Find real numbers $\alpha,\beta$ such that $a=\br(\alpha)$ and $b=\br(\beta)$. Since the point $b$ is smooth, the function $\br$ is differentiable at $\beta$. Write the derivative $\br'(\beta)$ as $x\cdot \br(0)+y\cdot\br'_\pm(0)$ for some real numbers $x,y$. Since $X$ is strictly convex and $0\ne b-a=\|b-a\|\cdot c$, the vector $\br'(\beta)$ is not parallel to the vector $c$ and hence $y\ne 0$. By Lemma~\ref{l:a}, $y>0$.

For every $\e\in[0,2]$ let 
$$\e_+=\min\{s\in [0,+\infty):\|\br(\beta+s)-\br(\beta)\|=\e\}$$and
$$\e_-=\max\{s\in (-\infty,0]:\|\br(\beta+s)-\br(\beta)\|=\e\}.$$
By Lemma~\ref{l:d}, $$\e_+=\e+o(\e)=\e+o(\e_+)\quad\mbox{and}\quad\e_-=-\e+o(\e)=-\e+o(\e_-).$$

If the point $c$ is not smooth, then $\br'_-(0)\ne \br'_+(0)$ and $\jj(0)<0$ by Lemma~\ref{l:j}. Then the number 
$$\delta=-\frac{\jj(s)\cdot y}{(1-\jjj(s)^2)}$$ is positive.

By Lemma~\ref{l:a}, for a small $\e>0$ we have
$$
\begin{aligned}
&\|\br(\beta+\e_+)-\br(\alpha)\|+\|\br(\beta+\e_-)-\br(\alpha)\|=\\
&=\|b-a\|+x\cdot\e_+-\frac{\jj(s)\cdot y\cdot|\e_+|}{1+\jjj(s)}+o(\e)+\|b-a\|+x\cdot\e_--\frac{\jj(s)\cdot y\cdot|\e_-|}{1-\jjj(s)}+o(\e)=\\
&=2\cdot\|b-a\|+x(\e_++\e_-)-\frac{\jj(s){\cdot} y{\cdot}|\e|(1+o(1))}{1+\jjj(s)}-\frac{\jj(s){\cdot} y{\cdot}|\e|(1+o(1))}{1-\jjj(s)}+o(\e)=\\
&=2\cdot\|b-a\|+o(\e)+2\delta|\e|+o(\e)=2\cdot\|b-a\|+(2\delta+o(1))|\e|.
\end{aligned}
$$and hence there exists $\e_0>0$ such that for any positive $\e<\e_0$ and the points $x=\br(\beta+\e_+)$ and $y=\br(\beta+\e_-)$ we have
$$\|x-a\|+\|y-a\|=\|\br(\beta+\e_+)-\br(\alpha)\|+\|\br(\beta+\e_-)-\br(\alpha)\|=2\cdot \|b-a\|+(2\delta+o(1))|\e|>2\cdot\|b-a\|+\delta|\e|.$$
The choice of $\e_+$ and $\e_-$ guarantees that $\|x-a\|=\e=\|y-a\|$.
This completes the proof of the ``only if'' part of the lemma.

\smallskip
To prove the ``only if'' part, assume that there exist positive $\delta$ and $\e_0$ such that for any positive $\e<\e_0$ there exist distinct points $x,y\in S_X$ such that $\|x-b\|=\e=\|y-b\|$ and $$\|x-a\|+\|y-a\|>2\|b-a\|+\delta\cdot\e.$$ We can assume that $\e_0$ is so small that for any positive $\e<\e_0$ the set $\{x\in S_X:\|x-b\|=\e\}$ coincides with the doubleton $\{\br(\beta+\e_+),\br(\beta+\e_-)\}$. Assuming that the point $c$ is smooth, we conclude that $\br'_-(0)=\br'_+(0)$ and hence $\jj(0)=0=\jjj(0)$. By Lemma~\ref{l:a}, 
$$\|\br(\beta+\e_+)-a\|+\|\br(\beta+\e_-)-a\|=2\cdot\|b-a\|+x\cdot(\e_++\e_-)+o(\e)=2\cdot\|b-a\|+o(\e).$$Replacing $\e_0$ by a smaller positive number, we can assume that 
$$\|\br(\beta+\e_+)-a\|+\|\br(\beta+\e_-)-a\|=2\cdot\|b-a\|+o(\e)<2\cdot\|b-a\|+\delta\cdot|\e|$$for all positive $\e<\e_0$. But this contradicts our assumption. This contradiction shows that the point $c$ is not smooth.
\end{proof}

\section{Special directions on the unit sphere}\label{s:special}

\begin{definition} Let $X$ be a Banach space. A point $x\in S_X$ is called {\em special} if for any bijective isometry $f:S_X\to S_Y$ to the unit sphere of a Banach space $Y$ and any points $y,z\in S_X$ with $y-z=\|y-z\|\cdot x$ we have 
$$f(y)-f(z)=\|f(y)-f(z)\|\cdot f(x)=\|y-z\|\cdot f(x).$$
\end{definition} 

\begin{proposition}\label{p1} Every non-smooth point of the unit sphere of a strictly convex $2$-dimensional Banach space is special.
\end{proposition}

\begin{proof} Let $c\in S_X$ be a non-smooth point of the sphere $S_X$. 
Since the Banach space $X$ is strictly convex, the set $c^\perp=\{x\in S_X:\{x\}=S_X\cap (x+\IR c)\}$ contains exactly two points. The complement $S_X\setminus c^\perp$ has exactly two connected components. Let  $A$ be the connected component of $S_X\setminus\ c^\perp$ containing the point $c$. It follows that $S_X=(-A)\cup p^\perp\cup A$. 

Let $\theta:A\to -A$ be the function assigning to each point $x\in A$ the unique point $y\in A$ such that $x-y=\|x-y\|\cdot c$. It is clear that the function $\theta$ is injective. 

Now take any bijective isometry $f:S_X\to S_Y$ of $S_X$ onto the unit sphere of a Banach space $Y$ and consider the continuous function 
$$g:A\to S_Y,\;g:x\mapsto \frac{f(x)-f(\theta(x))}{\|f(x)-f(\theta(x))\|}=\frac{f(x)-f(\theta(x))}{\|x-\theta(x)\|}.$$

To prove that the point $c$ is special, it suffices to check that the function $g$ is constant. To derive a contradiction, assume that the function $g$ is not constant. Then for some compact connected subset $K\subseteq A$ the image $g(K)$ is not a singleton. It is clear that $K$ is homeomorphic to the closed interval $[0,1]$. By a classical result of Sierpi\'nski \cite{S}, the unit interval $[0,1]$ cannot be covered by countably many pairwise disjoint closed proper subsets.

 Let $\Lambda_X,\Lambda_Y$ be the sets of non-smooth points on the spheres $S_X$ and $S_Y$, respectively. By Lemma~\ref{l:r}, the sets $\Lambda_X,\Lambda_Y$ are almost countable. For every $y\in \Lambda_Y$ consider the closed set $K_y=K\cap g^{-1}(y)$, which is not equal to $K$ as $g(K)\ne\{y\}$. Then the countable family $\mathcal K=\{K_y:y\in \Lambda_Y\}\cup\{\{x\}:x\in \Lambda_X\setminus\bigcup_{y\in\Lambda_Y}K_y\}$ consists of pairwise disjoint closed subsets of $K$. By the mentioned result of Sierpi\'nski, $K\ne\bigcup\mathcal K$. Consequently, there exists a point $b\in K\setminus\bigcup\K$. It follows that the image $g(b)$ is a smooth point of $S_Y$ and $b\notin \Lambda_X$, which means that the point $b$ is smooth. Let $a=\theta(b)\in -A$. The definition of the map $\theta$ ensures that $b-a=\|b-a\|\cdot c$.
 
Since the point $c$ is not smooth, we can apply Lemma~\ref{l:ns} and find positive $\delta$ and $\e_0$ such that for every positive $\e<\e_0$ there exist two distinct points $x,y\in S_X$ such that $\|x-b\|=\e=\|y-b\|$ and $\|x-a\|+\|y-a\|>2\|b-a\|+\delta\e$.

Since $f:S_X\to S_Y$ is an isometry and $b$ is a smooth point of the sphere $S_X$, its image $f(b)$ is a smooth point of the sphere $S_Y$, according to Lemma~\ref{l:ns}. Observe that $$0\ne f(b)-f(a)=f(b)-f(\theta(b))=\|b-\theta(b)\|\cdot g(b)=\|b-a\|\cdot g(a)=\|f(b)-f(a)\|\cdot g(a).$$
Since the point $g(a)\in S_X$ is smooth we can apply Lemma~\ref{l:ns} and find a positive $\e>0$ such that for any distinct points $u,v\in S_Y$ with $\|u-f(b)\|=\e=\|v-f(b)\|$ we have $\|u-f(a)\|+\|v-f(a)\|\not>2\cdot\|f(b)-f(a)\|+\delta\e$. 

By the choise of $\delta$ and $\e_0$, there exist  points $x,y\in S_X$ such that $\|x-b\|=\e=\|y-b\|$ and $\|x-b\|+\|y-b\|>2\cdot\|b-a\|+\delta\e$. Since $f:S_X\to S_Y$ is an isometry, for the points $u=f(x)$ and $v=f(y)$, we obtain
$$\|f(x)-f(b)\|=\|x-b\|=\e=\|y-b\|=\|f(y)-f(b)\|$$and
$$\|f(x)-f(a)\|+\|f(y)-f(a)\|=\|x-a\|+\|y-a\|>2\cdot\|b-a\|+\delta\e=2\cdot\|f(b)-f(a)\|+\delta\e,$$which contradicts the choice of $\e$. 
\end{proof}

\begin{definition} Let $x,y\in S_X$ be two points on the unit sphere of a 2-dimensional Banach space $X$. The pair $(x,y)$ is called
\begin{itemize}
\item {\em singular} if $\{x,-x\}\cap\{y,-y\}\ne\emptyset$ or there exists a point $z\in S_X$ such that\newline $S_X\cap (z+\IR x)=\{z\}=S_X\cap(z+\IR y)$;
\item {\em regular} if it is not singular.
\end{itemize}
\end{definition} 

\begin{lemma}\label{l:t} Let $u,v\in S_X$ be two points on the unit sphere of a strictly convex $2$-dimensional Banach space $X$. If the pair $(u,v)$ is regular, then for any point $w\in S_X$ there exist points $x,y,z\in S_X$ such that $$x\in z+\IR u\quad y\in z+\IR v\quad\mbox{and}\quad x-y\in \IR w.$$
\end{lemma}

\begin{proof} By the strict convexity of $X$, for every $z\in S_X$ there exist unique points $\vec u(z),\vec v(z)\in S_X$ such that $\{z,\vec u(z)\}=S_X\cap(z+\IR u)$ and $\{z,\vec v(z)\}=S_X\cap (z+\IR v)$. It is easy to see that the function 
$$\varphi:S_X\to X,\quad\varphi:z\mapsto \vec u(z)-\vec v(z),$$
is continuous. Since the pair $(u,v)$ is regular, $\varphi(z)\ne 0$ for any $z\in S_X$.
Then the function $$\psi:S_X\to S_X,\quad \psi:z\mapsto\frac{\varphi(z)}{\|\varphi(z)\|}=\frac{\vec u(z)-\vec v(z)}{\|\vec u(z)-\vec v(z)\|}$$is well-defined and continuous.

By the central symmetry of $S_X$, for every point $z\in S_X$ we have
$\vec u(-z)=-\vec u(z)$ and $\vec v(-z)=-\vec v(z)$, which implies $\psi(-z)=-\psi(z)$. 
Therefore, the image $\psi(S_X)$ of $S_X$ contains two opposite points. 
By the connectedness of $\psi(S_X)$, for any $w\in S_X$ there exists $z\in S_X$ such that $$w=\psi(z)=\frac{\vec u(z)-\vec v(z)}{\|\vec u(z)-\vec v(z)\|}.$$Then the points $x=\vec u(z)$, $y=\vec v(z)$ and $z$ have the required properties.
\end{proof} 

\begin{lemma}\label{l:reg} If the unit sphere of a strictly convex $2$-dimensional Banach space $X$ contains two linearly independent special points, then $S_X$ contains two special points $x,y$ such that the pair $(x,y)$ is regular.
\end{lemma}

\begin{proof} Let $x,y\in S_X$ be two linearly independent special points on the unit sphere of $X$. If the pair $(x,y)$ is regular, then we are done. So, assume that $(x,y)$ is singular. Then there exists a point $z\in S_X$ such that $S_X\cap(z+\IR x)=\{z\}=S_X\cap (z+\IR y)$. It follows that $z$ is a non-smooth point of the sphere $S_X$. By Proposition~\ref{p1}, the point $z$ is special. Taking into account that the Banach space $X$ is strictly convex, it can be shown that the pairs $(x,z)$ and $(y,z)$ are regular.

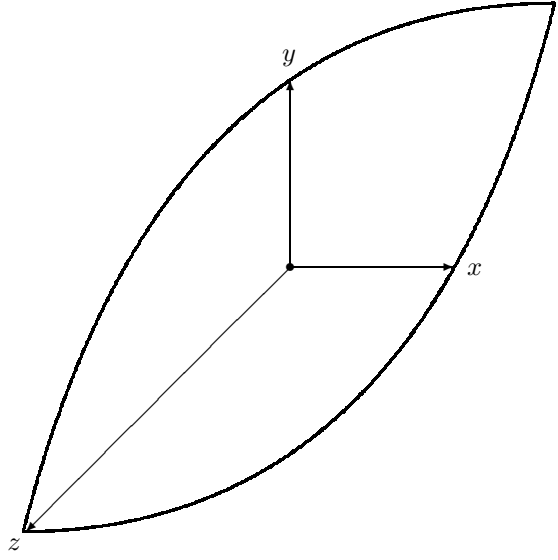
\begin{picture}(200,210)(-200,-105)
\put(0,0){\circle*{3}}
\put(0,0){\vector(0,1){71}}
\put(0,0){\vector(1,0){62}}
\put(0,0){\vector(-1,-1){100}}

\put(67,-3){$x$}
\put(-3,78){$y$}
\put(-107,-107){$z$}

\thicklines

\qbezier(-101,-100)(-50,100)(100,100)
\qbezier(-101,-100)(50,-100)(100,100)
\end{picture}
\end{proof}

The following helpful fact was proved by Tingley in \cite{Tingley}.

\begin{lemma}\label{l:Tingley} If $f:S_X\to S_Y$ is a bijective isometry between unit spheres of finite-dimensional Banach spaces, then $f(-x)=-f(x)$ for all $x\in S_X$.
\end{lemma}

\begin{proposition}\label{p2} A $2$-dimensional Banach space $X$ has the Mazur--Ulam property if its sphere contains two linearly independent special points.
\end{proposition}

\begin{proof} If $X$ is not strictly convex, then $X$ has the Mazur--Ulam property by Theorem~\ref{t:CS}. So, we assume that $X$ is strictly convex. Let $u,v\in S_X$ be two linearly independent special points on $S_X$. By Lemma~\ref{l:reg}, we can additionally assume that the pair $(u,v)$ is regular.  

Let $f:S_X\to S_Y$ be a bijective isometry of $S_Y$ onto the unit sphere $S_Y$ of an arbitrary Banach space $Y$. Let $L:X\to Y$ be a unique linear operator such that $L(u)=f(u)$ and $L(v)=f(v)$.  We claim that $L{\restriction}S_X=f$. Fix any point $w\in S_X$ and using Lemma~\ref{l:t}, find points $x,y,z\in S_X$ such that $x=z+au$, $y=z+bv$ and $x-y\in c w$ for some real numbers $a,b,c$. Replacing $x,y,z$ by $-x,-y,-z$, we can assume that $c>0$. Observe that $|a|=\|x-z\|$. If $a>0$, then by the special property of $u$, the equality $x=z+au=z+\|x-z\|u$ implies $f(x)=f(z)+\|x-z\|\cdot f(u)=a\cdot f(u)$. If $a<0$, then the equality $x=z+au=z-\|x-z\|u$ implies $z=x+\|x-z\|u$. By the special property of $u$, we have $f(z)=f(x)+\|x-z\|\cdot f(u)=f(x)-a\cdot f(u)$ and hence $f(x)=f(z)+a\cdot f(u)$. In both cases we obtain $f(x)=f(z)+a\cdot f(u)$. By analogy we can show that $f(y)=f(z)+b\cdot f(v)$. 
Then
$$f(x)-f(y)=(f(z)+a\cdot f(u))-(f(z)+b\cdot f(v))=a\cdot f(u)-b\cdot f(v)=a\cdot L(u)-b\cdot L(v).$$
It follows from $x-y=c\cdot w=\|x-y\|\cdot w$ that 
$$L(w)=\frac{L(x)-L(y)}{\|x-y\|}=\frac{L(z+au)-L(z+bv)}{\|x-y\|}=\frac{aL(u)-bL(v)}{\|x-y\|}=\frac{f(x)-f(y)}{\|x-y\|}$$and finally
$$\|L(w)\|=\frac{\|f(x)-f(y)\|}{\|x-y\|}=1.$$
Therefore, the linear operator $L:X\to Y$ is an isometry.

Consider the isometry $g=L^{-1}\circ f:S_X\to S_X$ and observe that $g(u)=u$ and $g(v)=v$. By Lemma~\ref{l:Tingley}, $g(-u)=-g(u)=-u$ and $g(-v)=-g(v)=-v$. Since the space $X$ is strictly convex, for any distinct points $x,y\in X$ and positive real numbers $a,b$ the intersection $(x+aS_X)\cap (y+bS_X)$ contains at most two distict points. This fact can be used to show that each point $x\in S_X$ is the unique point of the intersection $$(u+\|x-u\|S_X)\cap (-u+\|x+u\|S_X)\cap (v+\|x-v\|S_X)\cap(-v+\|x+v\|S_X)$$ which implies that $g(x)=x$ and hence $f=L{\restriction}S_X$.
\end{proof}

\section{Banach spaces with exactly two non-smooth points on the unit sphere}

Let $X$ be a strictly convex 2-dimensional Banach space whose unit sphere contains exactly two non-smoth points. Let $\mathbf e_1\in S_X$ be one of these non-smooth points.
Take any vector $\mathbf e_2\in X$ which is linearly independent with $\mathbf e_1$ and consider the natural parametrization $\br:\IR\to S_X$ of the 2-based Banach space $(X,\mathbf e_1,\mathbf e_2)$. For this parameterization we have $\br(0)=\mathbf e_1$ and $\br'_-(0)\ne \br'_+(0)$ as $\mathbf e_1=\br(0)$ is a non-smooth point of the unit sphere. Replacing the vector $\mathbf e_2$ by $\br'_\pm(0)=\frac12(\br'_-(0)+\br'_+(0))$, we can assume that $\mathbf e_2=\br'_\pm(0)$. 

We recall that $$\tfrac12(\br'_+(0)-\br'_-(0))=\jj(0)\cdot \br(0)+\jjj(0)\cdot \br'_\pm(0)=\jj(0)\cdot \mathbf e_1+\jjj(0)\cdot\mathbf e_2$$ and the numbers $\jj(0),\jjj(0)$ are called {\em radial} and {\em tangential jumps} of the derivative $\br'$ at zero. By Lemma~\ref{l:j}, $|\jjj(0)|<1$. We claim that those jumps are determined by the metric of the unit sphere. 

For every point $x\in S_X$ let $\overline x$ be the unique point of the sphere such that $\{x,\overline x\}=S_X\cap (x+\IR\mathbf e_1)$. The uniqueness of $\overline x$ follows from the strict convexity of $X$. 

\begin{lemma}\label{l:jj} 
$$\lim_{\e\to+0}\frac{\|\br(\e)-\br(0)\|}{\|\overline{\br(\e)}+\br(0)\|}=\frac{1-\jjj(0)}{1+\jjj(0)}\quad\mbox{and}\quad\lim_{\e\to+0}\frac{\|\br(\e)-\overline{\br(\e)}\|-2}{2\e}=\frac{\jj(0)}{1-\jjj(0)}.$$
\end{lemma}

\begin{proof} For a small positive $\e$, find a positive $\delta$ such that $\overline{\br(\e)}=-\br(-\delta)$. Observe that
\begin{multline*}
\br(\e)=\br(0)+\br'_+(0)\e+o(\e)=\mathbf e_1+(\jj(0)\mathbf e_1+(1+\jjj(0))\mathbf e_2)\e+o(\e)=\\
=(1+\jj(0)\e+o(\e))\mathbf e_1+(1+\jjj(0)+o(1))\e\mathbf e_2
\end{multline*}
and
\begin{multline*}
\br(-\delta)=\br(0)-\br'_-(0)\delta+o(\delta)=\mathbf e_1-(-\jj(0)\mathbf e_1+(1-\jjj(0))\mathbf e_2)\delta+o(\delta)=\\
=(1+\jj(0)\delta+o(\delta))\mathbf e_1-(1-\jjj(0)+o(1))\delta\mathbf e_2.
\end{multline*}
The equality $-\br(-\delta)=\overline{\br(\e)}$ implies $(1+\jjj(0)+o(1))\e=(1-\jjj(0)+o(1))\delta$ and $$\frac\e\delta=\frac{1-\jjj(0)+o(1)}{1+\jjj(0)+o(1)}.$$
Then
$$\lim_{\e\to+0}\frac{\|\br(\e)-\br(0)\|}{\|\overline{\br(\e)}+\br(0)\|}=\lim_{\e\to+0}\frac{\|(\br'_+(0)+o(1))\e)\|}{\|{-}\br(-\delta)+\br(0)\|}=\lim_{\e\to+0}\frac{\|\br'_+(0){+}o(1)\|\cdot|\e|}{\|(\br'_-(0)+o(1))\delta\|}=\lim_{\e\to+0}\frac{|\e|}{|\delta|}=\frac{1-\jjj(0)}{1+\jjj(0)}.$$

On the other hand,
$$
\begin{aligned}
&\|\br(\e)-\overline{\br(\e)}\|=\|\br(\e)+\br(-\delta)\|=\\
&=\|(1+\jj(0)\e+o(\e))\mathbf e_1+(1+\jjj(0)+o(1))\e\mathbf e_2+(1+\jj(0)\delta+o(\delta))\mathbf e_1-(1-\jjj(0)+o(1))\delta\mathbf e_2\|=\\
&=\|(2+\jj(0)(\e+\delta)+o(\e+\delta))\mathbf e_1\|=2+(\jj(0)+o(1))\e\Big(1+\frac{1+\jjj(0)+o(1)}{1-\jjj(0)+o(1)}\Big)=2+2\e\frac{\jj(0)+o(1)}{1-\jjj(0)}
\end{aligned}
$$
and hence
$$\lim_{\e\to+0}\frac{\|\br(\e)-\overline{\br(\e)}\|-2}{2\e}=\frac{\jj(0)}{1-\jjj(0)}.$$
\end{proof}

\begin{lemma}\label{l:xy} Let $s,\bar s\in\IR$ be two distinct real numbers such that $0\ne\br(s)-\br(\bar s)=\|\br(s)-\br(\bar s)\|\cdot\mathbf e_1$ and $\br(s)\notin\{\mathbf e_1,-\mathbf e_1\}$. Let $\br'(s)=x\mathbf e_1+y\mathbf e_2$ and $\br'(\bar s)=\overline x\mathbf e_1+\overline y\mathbf e_2$ for some real numbers $x,y,\overline x,\overline y$. For a small real number $\e$ let $\bar\e$ be the unique small real number such that $\br(\bar s+\bar\e)=\overline{\br(s+\e)}$. Then
\begin{enumerate}
\item $y>0$ and $\overline y<0$;
\item $\lim\limits_{\e\to 0}\dfrac{\|\br(\bar s+\bar\e)-\br(\bar s)\|}{\|\br(s+\e)-\br(s)\|}=-\dfrac{y}{\overline y}$;
\item $\lim\limits_{\e\to0}\dfrac{\|\br(s+\e)-\br(\bar s+\bar \e)\|-\|\br(s)-\br(\bar s)\|}{\e}=x-\overline x\cdot\dfrac{y}{\overline y}$;
\item $\lim\limits_{\e\to+0}\dfrac{\|\br(s+\e)-\br(\bar s)\|-\|\br(s)-\br(\bar s)\|}{\e}=x-\dfrac{\jj(0)\cdot y}{1+\jjj(0)}$;
\item $\lim\limits_{\e\to-0}\dfrac{\|\br(s+\e)-\br(\bar s)\|-\|\br(s)-\br(\bar s)\|}{\e}=x+\dfrac{\jj(0)\cdot y}{1-\jjj(0)}$;
\item The numbers $x,y,\overline x,\overline y$ are uniquely determined by the equations \textup{(2)--(5)}.
\end{enumerate}
\end{lemma}

\begin{proof}  It follows from $0\ne\br(s)-\br(\bar s)=\|\br(s)-\br(\bar s)\|\cdot\mathbf e_1$ that $y>0$ and $\overline y<0$.

For a small number $\e$ we have
$$\br(s+\e)-\br(s)=(\br'(s)+o(1))\e=(x+o(1))\e\mathbf e_1+(y+o(1))\e\mathbf e_2$$and
$$\br(\bar s+\bar\e)-\br(\bar s)=\br'(\bar s)\bar\e+o(\bar\e)=(\overline x+o(1))\bar\e\mathbf e_1+(\overline y+o(1))\bar\e\mathbf e_2.$$
The equality $\overline{\br(s+\e)}=\br(\bar s+\bar\e)$ implies $$(y+o(1))\e=(\overline y+o(1))\bar\e$$ and then
$$\lim_{\e\to 0}\frac{\|\br(\bar s+\bar\e)-\br(\bar s)\|}{\|\br(s+\e)-\br(s)\|}=\lim_{\e\to0}\frac{\|\br'(\bar s)+o(1)\|\cdot|\bar\e|}{\|\br'(s)+o(1)\|\cdot|\e|}=\lim_{\e\to0}\frac{|\bar \e|}{|\e|}=\lim_{\e\to0}\frac{|y+o(1)|}{|\overline y+o(1)|}=\frac{|y|}{|\overline y|}=-\frac{y}{\overline y}.$$

Also
$$
\begin{aligned}
&\br(s+\e)-\br(\bar s+\bar\e)=\br(s)+\br'(s)\e+o(\e)-\big(\br(\bar s)+\br'(\bar s)+o(\bar\e)\big)=\\
&=(\br(s)-\br(\bar s))+(x+o(1))\e\mathbf e_1+(y+o(1))\e\mathbf e_2-(\overline x+o(1))
\bar\e\mathbf e_1-(\overline y+o(1))\bar \e\mathbf e_2=\\
&=\|\br(s)-\br(\bar s)\|\mathbf e_1+(x\e-\overline x\bar\e+o(\e+\bar \e))\cdot\mathbf e_1=\\
&=\big(\|\br(s)-\br(\bar s)\|+(x-\overline x\cdot\frac{y}{\overline y}+o(1))\e\big)\cdot\mathbf e_1
\end{aligned}
$$
and hence
$$\lim_{\e\to0}\frac{\|\br(s+\e)-\br(\bar s+\bar \e)\|-\|\br(s)-\br(\bar s)\|}{\e}=x-\overline x\cdot\frac{y}{\overline y}.
$$
By Lemma~\ref{l:a}, 
$$\|\br(s+\e)-\br(\bar s)\|=\|\br(s)-\br(\bar s)\|+x\e-\frac{\sign(\e)\cdot\jj(0)\cdot y}{1+\sign(\e)\jjj(0)}\e+o(\e)$$and hence
$$\lim_{\e\to+0}\frac{\|\br(s+\e)-\br(\bar s)\|-\|\br(s)-\br(\bar s)\|}{\e}=x-\frac{\jj(0)\cdot y}{1+\jjj(0)}$$and
$$\lim_{\e\to-0}\frac{\|\br(s+\e)-\br(\bar s)\|-\|\br(s)-\br(\bar s)\|}{\e}=x+\frac{\jj(0)\cdot y}{1-\jjj(0)}.$$

Therefore, the items (1)--(5) of Lemma~\ref{l:xy} are proved.
\smallskip

The equations (4),(5) determine the numbers $x,y$ uniquely because 
$$\left|\begin{array}{cc}
1&-\frac{\jj(0)\cdot y}{1+\jjj(0)}\\
1&\frac{\jj(0)\cdot y}{1-\jjj(0)}
\end{array}
\right|=\frac{\jj(0)\cdot y}{1-\jjj(0)}+\frac{\jj(0)\cdot y}{1+\jjj(0)}=\frac{2\cdot\jj(0)\cdot y}{1-\jjj(0)^2}\ne0.
$$
The equation (2) allows us to find $\overline y$ and then $\overline x$ can be found from the equation (3).
\end{proof}

\section{Proof of Theorem~\ref{t:main}}

Given any non-smooth $2$-dimensional Banach space, we should prove that $X$ has the Mazur--Ulam property. If $X$ is not strictly convex, then $X$ has the Mazur--Ulam property by Theorem~\ref{t:CS}. If the sphere $S_X$ contains more that two non-smooth points, then $X$ has the Mazur--Ulam property by Propositions~\ref{p1} and \ref{p2}. So, we assume that $X$ is strictly convex and $S_X$ contains exactly two non-smooth points. Let $\mathbf e_1$ be one of them. Then $-\mathbf e_1$ is the other non-smooth point of $X$.

Take any vector $\mathbf e_2\in X\setminus(\IR\cdot\mathbf e_1)$ and consider the natural parameterization $\br:\IR\to X$ of the $2$-based Banach space $(X,\mathbf e_1,\mathbf e_2)$. Since $\mathbf e_1$ is a non-smooth point of $S_X$, the one-sided derivatives $\br'_-(0)$ and $\br'_+(0)$ are distinct. Replacing the vector $\mathbf e_2$ by $\br'_\pm(0)$, we can assume that $\mathbf e_2=\br'_\pm(0)$. Let $L=\min\{s\in[0,\infty):\br(s)=-\mathbf e_1\}$ be the half-length of the sphere $S_X$,  and $\breve S_X=\br([0,L])$ be the upper half-sphere of $X$. By Lemmas~\ref{l:r}(5) and \ref{l:intrinsic}, the restriction $\br{\restriction}_{[0,L]}:[0,L]\to \breve S_X$ in an isometry of $[0,L]$ onto the half-sphere $\breve S_X$ endowed with the intrinsic metric.

To show that the space $X$ has the Mazur--Ulam property, fix any bijective isometry $f:S_X\to S_Y$ of $S_X$ onto the unit sphere of an arbitary Banach space $Y$. It is clear that the space $Y$ is $2$-dimensional. Lemma~\ref{l:ns} implies that $\tilde {\mathbf e}_1=f(\mathbf e_1)$ and $-\tilde{\mathbf e}_1$ are unique non-smooth points of the sphere $S_Y$. Repeating the above argument, we can find a vector $\tilde{\mathbf e}_2\in Y\setminus(\IR\cdot \tilde{\mathbf e}_1)$ such that for the natural parameterization $\tilde \br:\IR\to Y$ of the $2$-based Banach space $(Y,\tilde{\mathbf e}_1,\tilde{\mathbf e}_2)$ we have $\tilde \br'_\pm(0)=\tilde{\mathbf e}_2$. Replacing $\tilde{\mathbf e}_2$ by $-\tilde{\mathbf e}_2$ we can additionally assume that $\tilde\br(\e)\in f(\breve S_X)$ for small positive numbers $\e$. 

 Then $f(\breve S_X)$ coincides with the half-sphere $\breve S_Y$ of the $2$-based Banach space $(Y,\tilde{\mathbf e_1},\tilde{\mathbf e}_2)$. Since $\breve S_X$ is isometric to $\breve S_Y$, the spheres $S_X$ and $S_Y$ have the same half-length. By Lemma~\ref{l:intrinsic}, the restriction $\tilde\br{\restriction}_{[0,L]}:[0,L]\to \breve S_Y$ is an isometry of $[0,L]$ onto the half-sphere $\breve S_Y$ endowed with the intrinsic metric. The isometry $f{\restriction}_{\breve S_X}:\breve S_X\to\breve S_Y$ remains an isometry with respect to the intrinsic metrics on the half-spheres. Then $(\tilde\br{\restriction}_{[0,L]})^{-1}\circ f\circ\br{\restriction}_{[0,L]}$ is an isometry of $[0,L]$ which does not move zero and hence is an identity map of $[0,L]$. Consequently, $\tilde\br(s)=f\circ\br(s)$ for all $s\in[0,L]$. Using this fact and Lemma~\ref{l:r}(1), we can show that $\tilde \br(s)=f\circ \br(s)$ for all $s\in\IR$.
 
 Let $I:X\to Y$ be the linear operator such that $I(\mathbf e_1)=\tilde{\mathbf e}_1$ and $I(\mathbf e_2)=\tilde{\mathbf e}_2$. Applying Proposition~\ref{p1} and Lemmas~\ref{l:Tingley} and \ref{l:jj}, we can show that the spheres $S_X$ and $S_Y$ have the same radial and tangential jumps $\jj(0)$ and $\jjj(0)$.   
Lemma~\ref{l:xy} implies that $I(\br'(s))=\tilde \br'(s)$ for every $s\in (0,L)$ with $\br(s)\ne\overline{\br(s)}$. Since $X$ is strictly convex, the set $\{s\in[0,L]:\br(s)=\overline{\br(s)}\}$ is a singleton. Now the continuity of the functions $\br'$ and $\tilde\br'$ on $(0,L)$ implies that $I(\br'(s))=\tilde \br'(s)$ for all $s\in(0,L)$. Since $\br$ is continuously differentiable on $[0,L]$, for every $s\in[0,L]$ we have $\br(s)=\br(0)+\int_0^s\br'(t)\,dt$ and hence $$I(\br(s))=I(\br(0))+\int_0^sI(\br'(t))\,dt=\tilde\br(0)+\int_0^s\tilde\br'(t)dt=\tilde \br(s)=f\circ\br(s).$$
By Lemmas~\ref{l:Tingley} and \ref{l:r}(1), $$I(\br(s+L))=I(-\br(s))=-I(\br(s))=-f(\br(s))=f(-\br(s))=f(\br(s+L))$$ for every $s\in[0,L]$ and hence $I\circ\br{\restriction}_{[L,2L]}=\tilde\br{\restriction}_{[L,2L]}$. Therefore, $I$ is a linear operator extending the isometry $f$. The equality $I(S_X)=S_Y$ implies $I(B_X)=B_Y$, which means that $I$ is a linear isometry of the Banach spaces $X,Y$. 

\section*{Acknowledgements}
The second author has been partially supported by
Junta de Extremadura programs GR-15152 and IB-16056 and DGICYT
projects MTM2016-76958-C2-1-P and PID2019-103961GB-C21 (Spain).

\end{document}